%% file: T_adic_affinoid.tex
\documentclass{amsart}
\usepackage{amsmath, amsfonts}
\usepackage{amssymb}
\usepackage{mathtools}
\usepackage{tikz-cd}

\usepackage{amsthm}
\usepackage{upgreek}
\usepackage{hyperref}
\usepackage{float, textcomp}
\usepackage{hyperref}
\usepackage{chngpage}
\usepackage{graphicx}
\usepackage{float}
\usepackage{caption}

\usepackage{cleveref}

\theoremstyle{plain}
\newtheorem{theorem}{Theorem}

\newtheorem{prop}[theorem]{Proposition}
\newtheorem{lemma}[theorem]{Lemma}
\newtheorem{corollary}[theorem]{Corollary}

\theoremstyle{definition}
\newtheorem{definition}[theorem]{Definition}

\numberwithin{theorem}{section} 

\theoremstyle{remark}
\newtheorem*{remark}{Remark}

\DeclareMathOperator{\Tr}{Tr}

\DeclareMathOperator{\Gal}{Gal}

\DeclareMathOperator{\Sym}{Sym}
\DeclareMathOperator{\sgn}{sgn}

\newcommand{\Fq}{\mathbb{F}_q}

\newcommand{\Cp}{\mathbb{C}_p}

\newcommand{\Qp}{\mathbb{Q}_p}
\newcommand{\Qq}{\mathbb{Q}_q}
\newcommand{\Zp}{\mathbb{Z}_p}
\newcommand{\Zq}{\mathbb{Z}_q}

\newcommand{\Zptimes}{\mathbb{Z}_p^\times}
\newcommand{\Zqtimes}{\mathbb{Z}_q^\times}

\DeclareMathOperator{\ordp}{ord}

\newcommand{\ord}[1]{\ordp_p{#1}}

\newcommand{\floor}[1]{\left \lfloor #1 \right \rfloor}
\newcommand{\ceil}[1]{\left \lceil #1 \right \rceil}

\date{\today}
\title{ $T$-adic Exponential Sums over Affinoids }
\author{Matthew Schmidt}
\email{mwschmid@buffalo.edu}
\address{Department of Mathematics, SUNY Buffalo}

\usepackage{amsfonts}
\usepackage[section,sort=standard]{glossaries}
\usepackage{glossary-mcols}

\setglossarystyle{mcolindex}

\newglossary[op-glg]{notation}{op-gls}{op-glo}{Notation}
\loadglsentries{glossary}

\makeglossaries

\begin{document}

\begin{abstract}
We introduce and develop $(\pi,p)$-adic Dwork theory for $L$-functions of exponential sums associated to one-variable rational functions, interpolating  $p^k$-order exponential sums over affinoids.  Namely, we prove a generalization of the Dwork-Monsky-Reich trace formula and apply it to establish an analytic continuation of the $C$-function $C_f(s,\pi)$.  We compute the lower  $(\pi,p)$-adic bound, the Hodge polygon, for this $C$-function. Along the way, we also show why a strictly $\pi$-adic theory will not work in this case. 
\end{abstract}

\maketitle 
\tableofcontents

\section{Introduction}

Let $p$ be a prime and $q=p^a$, some integer $a\geq 1$. Fix $\ell\geq 1$ distinct elements $P_1, \cdots, P_\ell\in\Fq\cup\{\infty\}$. Without loss of generality, take $P_1=\infty$ and $P_2=0$, assuming $\ell\geq 2$ for the rest of the paper.
For $x\in\Fq$, denote by $\widehat{x}$ the Teichm\"uller lift of $x$ in $\Zq$. 

Let $E(x)$ be the Artin-Hasse exponential series, $T$ a formal variable and $\pi$ such that $E(\pi)=1+T$. To $f(x)= \sum_{j=1}^\ell\sum_{i=1}^{d_j} \frac{a_{ij}}{(x-\widehat{P}_j)^i}\in \Zq[\frac{1}{x-\widehat{P}_1}, \cdots, \frac{1}{x-\widehat{P}_\ell}]$, $a_{d_j,j}\neq 0$, we associate a $\pi$-adic exponential sum\footnote{The literature (\cite{LiuWan}, \cite{LiuLiuNiu}, \cite{Li}, etc.) generally deals with $T$-adic exponential sums, but for convenience, we will do things $\pi$-adically. There is no difference and our results can be stated either way.}: 
\begin{align}\label{exp_sum_intro}
S_f(k,\pi) = \sum_{\substack{x\in\widehat{\mathbb{F}_{q^k}^\times},\\ x\neq\widehat{P}_1,\cdots, \widehat{P}_\ell }}(E(\pi))^{\Tr_{\mathbb{Q}_{q^k}/\Qp}({f}(x))},
\end{align}
and we say the characteristic function, or $C$-function, attached to this exponential sum is
\[
	C_f(s, \pi) =\exp(\sum_{k=1}^\infty -(q^k-1)^{-1} S_f(k, \pi)\frac{s^k}{k}).
\]

When $T=\zeta_p-1$, $\zeta_p$ a primitive $p$th root of unity, (\ref{exp_sum_intro}) becomes the exponential sum over a one-dimensional affinoid studied by Robba  in \cite{Robba} and Zhu in \cite{Zhu}. Oppositely, letting $\zeta_{p^m}$  be $p^m$th roots of unity and $T=\zeta_{p^m}-1$ yields exponential sums of $p^m$-order over one-dimensional affinoids. In the classical case, these $p^m$-order exponential sums were studied by Liu and Wei in \cite{LiuWei}.  The purpose of $\pi$-adic (and $(\pi,p)$-adic) theory is to interpolate all of these exponential sums in a single $C$-function. Whenever we set $\pi$ to be a value in $c\in\Cp$, we say we specialize at $\pi=c$.

When $f(x)$ has one or two poles, Liu and Wan (\cite{LiuWan}) built a $T$-adic Dwork theory and computed, among other things, a Hodge polygon for this $C$-function. In this paper, we extend their results to the case when $\ell\geq 3$ by generalizing the affinoid Dwork theory used earlier by Zhu in \cite{Zhu}. The bulk of our work is lifting this Dwork theory to the $\pi$-adic case. That is, we construct a Banach module  $\mathcal{Z}^\pi$ and a completely continuous operator $\alpha_a$ on $\mathcal{Z}^\pi$ such that
\[
	C_f(s,\pi)  = \det(1-\alpha_a s).
\]

Unlike Liu and Wan's case, however, a purely $T$-adic theory is not precise enough. When $\ell \geq 3$, the $\alpha_a$ operator is not $\pi$-adically completely continuous and we cannot apply Dwork theory (see Corollary~\ref{not_cc}). To resolve this, we utilize the $(\pi,p)$-adic norm, used for the same reason by Li in \cite{Li}, to produce sharper estimates and make $\alpha_a$ completely continuous. 

Our main result, the computation of the $(\pi,p)$-adic Hodge polygon, is as follows:
For $k=1,\cdots,\ell$, let $\mathrm{HP}_k^c$ be the Newton polygon with vertices $$\{(n, \frac{a(p-1)n(n-1)}{2d_k}c)\}_{n\geq 0},$$ where $c$ is a real number with $0< c\leq \frac{1}{p-1}$.

We define the  $(\pi^{1/c}, p)$-adic Hodge polygon, $\mathrm{HP}^c$, to be the concatenation of $\mathrm{HP}_1^c, \cdots, \mathrm{HP}_\ell^c$. 

\begin{theorem}\label{main_theorem}
The $(\pi^{1/c}, p)$-adic Newton polygon of $C_f(s,\pi)$ lies above $\mathrm{HP^c}$.
\end{theorem}

As an example, consider the case where $\pi_1$ is a root of $\log(E(x))$ with $\ord\pi_1=1/(p-1)$. After specializing at $\pi=\pi_1$, Theorem~\ref{main_theorem} implies, taking $c=\frac{1}{p-1}$, that the corresponding Hodge polygon is nothing but the concatenation of:
	$$\{(n, \frac{an(n-1)}{2d_k})\}_{n\geq 0},$$
over $k=1,\cdots, \ell$, and this is exactly the same Hodge bound obtained in \cite{Zhu}.

Our construction of a $(\pi,p)$-adic theory opens up many avenues of future development. Liu, Liu and Niu in \cite{LiuLiuNiu}, for instance, compute the generic Newton polygon for the classical $T$-adic $C$-function, and there is a natural question as to whether their results can be extended to the affinoid case. Similarly, Ren, Wan, Xiao and Yu in (\cite{RWXY}) considered exponential sums over higher rank Artin-Schreier-Witt towers and Liu and Liu in \cite{LiuLiu} studied twisted $T$-adic exponential sums. Extending both of these results to the affinoid case might be interesting. 

This paper was written under the supervision of my advisor, Hui June Zhu. I thank her for her constant advice and guidance.   

%%%%%%%%%%%%%%%%%%%%%%%%%%%%%%%%%%%%%%%%%%
%
% Preliminaries
%
%%%%%%%%%%%%%%%%%%%%%%%%%%%%%%%%%%%%%%%%%%
\section{Preliminaries}

We will need some results about Tate and Banach algebras. For a more comprehensive review, see \cite{BGR}, \cite{bosch2} and \cite{Coleman}. 

\subsection{Tate Algebras}

Let $(A, |\cdot|)$ be an ultrametrically normed ring. Define the Tate algebra over $A$ to be 
\begin{align*}
	\gls{TateAlgebra} = \{\sum_{i_1,\cdots,i_n\in\mathbb{Z}_{\geq 0}} a_{i_1, \cdots, i_n}X_1^{i_1}\cdots X_n^{i_n} \in A[[X_1, \cdots, X_n]] : |a_{i_1, \cdots, i_n}|\to 0\\\textrm{ as } i_1+\cdots + i_n\to\infty\},
\end{align*}
and equip $A\langle X_1, \cdots, X_n\rangle$ with the gauss norm \gls{GaussNorm}:
$$|\sum_{i_1,\cdots,i_n\in\mathbb{Z}_{\geq 0}} a_{i_1, \cdots, i_n}X_1^{i_1}\cdots X_n^{i_n} |_{gauss}=\sup_{i_1,\cdots,i_n} |a_{i_1,\cdots,i_n}|.$$

\subsection{Banach algebras and modules}

Let $A$ be a complete unital commutative ring separated with respect to a non-trivial ultrametric norm $|\cdot |$ such that 
\begin{enumerate}
\item $|1|=1$
\item $|a+b|\leq \max\{|a|,|b|\}$
\item $|ab|\leq |a||b|$
\item $|a|=0$ if and only if $a=0$,
\end{enumerate}
for all $a,b\in A$. 

We call $A$ a Banach algebra. Moreover, if $E$ is an ultrametrically normed complete module over $A$ such that $|ae|\leq |a||e|$ for $a\in A$ and $e\in E$, we say $E$ is a Banach module over $A$. A Banach module $E$ over $A$ has an orthonormal basis $\{e_i\}_{i\in I}\subset E$ if for each $x\in E$ we can write uniquely
$x=\sum_{i\in I} a_i e_i$ 
for $a_i\in A$ with $|a_i|\to 0$ as $i\to\infty$. 

For a bounded Banach module operator $\phi: B\to C$, we write the standard operator norm \gls{OperatorNorm}:
$$\|\phi\|_{op}=\sup_{b\in B, |b|=1}|\phi(b)|.$$
If $\{e_i\}_{i\in I}$ is an  orthonormal bases for $B$, then an endomorphism of $B$, $\phi$, is completely continuous if 
	$$\lim_{i\to \infty} \sup_{j\in I}|b_{ij}|=0,$$
where $\phi(e_i)=\sum_{j\in I}b_{ij}e_j$.

\section{$p$-adic Spaces}

Once and for all, fix $0<r<1$ and $R\in\Cp$ with $|R|_p=r$ and let $s$ be a $p$-power. 
 Define $\gls{Hrs} = \Cp\langle \frac{R}{x-\widehat{P_1^s}}, \cdots, \frac{R}{x-\widehat{P_\ell^s}}\rangle$ to be the Tate algebra of rigid analytic functions over an affinoid $\mathbb{A}_{r,s}=\{x\in\Cp : |x|_p\leq 1/r, |x-\widehat{P_j^s}|_p\geq r\textrm{ for } 2\leq j\leq \ell\}$ with  
 supremum norm \gls{HrsNorm}:
	$$\|\xi\|_{r,s} = \sup_{x\in \mathbb{A}_{r,s}} |\xi(x)|_p.$$
	
\begin{remark}\label{gauss_sup}
Let $A$ be any algebraically closed and ultrametrically normed field and consider the Tate algebra $A\langle X_1, \cdots, X_n\rangle$. It is well known that if $Z = \{(x_1, \cdots, x_n) \in A^n : |x_i|\leq 1\}$, then for $f\in A\langle X_1, \cdots, X_n\rangle$,
	$$\sup_{(x_1, \cdots, x_n)\in Z} |f(x_1, \cdots, x_n)| = |f|_{gauss}.$$
However in the above, when $X_i=\frac{R}{x-\widehat{P_i^s}}$, we see that $(X_1, \cdots, X_n)\in Z$ if and only if, $i\neq 1$, $|\frac{R}{x-\widehat{P_i^s}}|_p\leq 1$, which implies $|x-\widehat{P_i^s}|_p\geq r$, and for $i=1$, $|Rx|_p\leq 1$, which yields $|x|_p\leq 1/r$. Hence $Z=\mathbb{A}_{r,s}$ and $|\cdot|_{gauss}=\|\cdot\|_{r,s}$ on $\mathcal{H}_{r,s}$. 
\end{remark}

$\mathcal{H}_{r,s}$ has two important orthonormal bases that we will utilize. 

\begin{prop}\label{padic_ortho}
The set 
	$$\left\{\left (\frac{R}{x-\hat{P_j^s}}\right)^i\right\}_{\substack{1\leq j\leq \ell \\ 0\leq i}}$$
forms an orthonormal basis for $\mathcal{H}_{r,s}$ over $\Cp$. (When convenient, we will use the notation $\gls{Bij}=\frac{1}{(x-\widehat{P_j})^i}$.)
\end{prop}
\begin{proof}
See Lemma 2.1 and the comment following its proof on p.1535 in \cite{Zhu}.
\end{proof}

\begin{prop}\label{padic_reich}
Let $v(x) = (x-\widehat{P}_1)\cdots  (x-\widehat{P}_\ell)$. The set
	$$\left\{\frac{x^i}{R^{i-j\ell}v^j}\right\}_{\substack{i\geq 0, (i, \ell)=1,\\ j\geq 0}}$$
forms an orthonormal basis for $\mathcal{H}_{r,s}$ over $\Cp$. 
\end{prop}
\begin{proof}
See Theorem 2 and the remark following it in \cite{Reich}.
\end{proof}

%%%%%%%%%%%%%%%%%%%%%%%%%%%%%%%%%%%%%%%%%%
%
% (\pi,p)-adic Spaces
%
%%%%%%%%%%%%%%%%%%%%%%%%%%%%%%%%%%%%%%%%%%
\section{$(\pi,p)$-adic Spaces}

Let $\pi$ be a formal variable.

\begin{definition}\label{pi_p_norm}
For $f(\pi)=\sum_{i=0}^\infty b_i\pi^i\in\Zq[[\pi]]$, define the $(\pi, p)$-norm on $\Zq[[\pi]]$ \gls{pipNorm}:
	$$|\sum_{i=0}^\infty b_i\pi^i|_{\pi, p} = \max_i|b_i|_pp^{-i}.$$
\end{definition}

\begin{lemma} \label{Tadic_norm}
 $|\cdot|_{\pi,p}$ is a complete multiplicative norm on $\Zq[[\pi]]$. 
\end{lemma}
\begin{proof}
Let $f(x)=\sum_{i=0}^\infty b_i\pi^i, g(x)=\sum_{i=0}^\infty c_i\pi^i\in\Zq[[\pi]]$. The only nontrivial thing to prove $|\cdot|_{\pi, p}$ is a norm is $|f+g|_{\pi, p}\leq \max(|f|_{\pi,p}, |g|_{\pi,p}).$ Then: 
\begin{align*}
	|f+g|_{\pi, p}&=\max_i(|b_i+c_i|_pp^{-i})\leq \max_i(\max(|b_i|_p, |c_i|_p)p^{-i}) \\
		&= \max(\max_i |b_i|_pp^{-i},\max_i |c_i|_pp^{-i}) = \max(|f|_{\pi,p}, |g|_{\pi,p}).
\end{align*}
To see that $\Zq[[\pi]]$ is complete with respect to this norm, observe that $|\cdot|_{\pi,p}$ is just the norm induced by the $(\pi, p)$-topology on $\Zq[[\pi]]$, and 
	$$\varprojlim_i \Zq[[\pi]]/(\pi,p)^i\cong \Zq[[\pi]].$$
One direction of the inequality to show $|\cdot|_{\pi,p}$ is multiplicative is clear:
\begin{align*}
	|fg|_{\pi,p}&=\max_i|\sum_{\substack{j+k=i\\j,k\geq 0}} b_jc_k|_pp^{-i}
		\leq\max_i\max_{j,k}\left (|b_j|_pp^{-j}\cdot |c_k|_pp^{-k}\right)\leq |f|_{\pi,p}|g|_{\pi,p}.
\end{align*}
For the opposite inequality, let $i_0$ and $j_0$ be the minimal integers such that $|f|_{\pi,p}=|b_{i_0}|_pp^{-i_0}$ and $|g|_{\pi,p}=|c_{j_0}|_pp^{-j_0}$. If we write $fg=\sum_{i=0}^\infty a_i\pi^i$, then 
	$$|a_{i_0+j_0}|_p=|b_{i_0}c_{j_0}+\sum_{\substack{i+j=i_0+j_0\\i,j\geq 0, i\neq i_0, j\neq j_0 }} b_ic_j|_p.$$ 
Take some $i,j$, $i\neq i_0$ and $j\neq j_0$, with $i+j=i_0+j_0$ so that either $i<i_0$ and $j>j_0$ or $j<j_0$ and $i>i_0$. In either case, by the minimality of $i_0$ and $j_0$, $|b_i|_p|c_j|_p<|b_{i_0}|_p|c_{j_0}|_p$, and so $|a_{i_0+j_0}|_p=|b_{i_0}c_{j_0}|_p$. Hence:
	$$|fg|_{\pi,p}=\max_i |a_i|_pp^{-i}\geq |a_{i_0+j_0}|_pp^{-(i_0+j_0)}=|f|_{\pi,p}|g|_{\pi,p}.$$
\end{proof} 

Because both $\Zq[[\pi]]$ and $\mathcal{H}_{r,s}$ are Banach modules over $\Zq$, we can consider the following completed tensor product of $\Zq$-Banach modules (again see \cite{Coleman}, p.424):
\begin{definition}
Define a module 
	$$\gls{Hrspi} = \Zq[[\pi]]{\hat{\otimes}}_{\Zq} \mathcal{H}_{r,s}$$
equipped with the norm coming from the completed tensor product \gls{HrsNorm}:
$$\| n\|_{r,s}=\inf\sup_i |b_i(\pi)|_{\pi,p}\|\xi_i\|_{r,s},$$
where the infimum is taken over all representations of $n=\sum_ib_i(\pi)\otimes\xi_i$, with $|b_i(\pi)|_{\pi, p}\|\xi\|_{r,s}\to 0$ as $i\to\infty$.

Note that  for the sake of notation when referring to simple tensors in $\mathcal{H}_{r,s}^\pi$ we will just write $a\otimes b$ rather than $a\hat{\otimes}b$.
\end{definition}

\begin{prop} 
 For $g,h\in\mathcal{H}_{r,s}^\pi$, $\|gh\|_{r,s}\leq \|g\|_{r,s}\|h\|_{r,s}$.
\end{prop}
\begin{proof}
For $g,h\in\mathcal{H}_{r,s}^\pi$ with arbitrary representations $g=\sum_i b_i\otimes g_i$ and $h=\sum_i c_i\otimes h_i$, 
$$gh = (\sum_i b_i\otimes g_i)(\sum_j c_j\otimes h_j)=\sum_{i,j}b_ic_j\otimes g_ih_j.$$
Hence by Lemma~\ref{Tadic_norm} and the fact that the norm on the Tate algebra is multiplicative,
\begin{align*}
	\|gh\|_{r,s}&=\inf_{gh=\sum_ie_i\otimes\xi_i}\sup_{i}|e_i|_{\pi,p}\|\xi_i\|_{r,s}\leq \inf_{\substack{g =\sum_ib_i\otimes g_i\\h=\sum_jc_j\otimes h_j}}\sup_{i,j}|b_ic_j|_{\pi,p}\|g_ih_j\|_{r,s}\\
	&\leq \inf_{\substack{g =\sum_ib_i\otimes g_i\\h=\sum_jc_j\otimes h_j}}\sup_{i,j}(|b_i|_{\pi,p}\|h_j\|_{r,s})(|c_j|_{\pi,p}\|g_i\|_{r,s}) = \|g\|_{r,s}\|h\|_{r,s}.
\end{align*}
\end{proof}

Let $\gls{C}= \Zq[[\pi]]\hat{\otimes}_{\Zq} \Cp$ and define a $\mathcal{C}$-module structure on $\mathcal{H}_{r,s}^\pi$ in the following way: for a tensor $b{\otimes}\xi\in\mathcal{H}_{r,s}^\pi$ and a tensor $b'{\otimes} \xi'$ in $\mathcal{C}$, 
	$$(b{\otimes} \xi)(b'{\otimes}\xi')=bb'{\otimes}\xi\xi',$$
and extend linearly. The $\Zq$-Banach module $\mathcal{C}$ also has an induced tensor product norm defined similarly to the above. Abusing notation, we will write it as $\|\cdot \|_{\pi,p}$.

\begin{prop}\label{structure_prop}
$\mathcal{H}_{r,s}^\pi$ is a $ \mathcal{C}$-Banach module and if $\{e_i\}_{i\in I}$ is an orthonormal basis for $\mathcal{H}_{r,s}$ over $\Cp$ then $\{1\otimes e_i\}_{i\in I}$ is an orthonormal basis for $\mathcal{H}_{r,s}^\pi$ over $\mathcal{C}$.
\end{prop}
\begin{proof}
The first statement is clear; see Section 3.1.1 in \cite{BGR} to prove that this multiplication is well-defined.

For the second statement, by Proposition 3 in Appendix B of \cite{bosch2} and a basic identity about completed tensor products, there is an isomorphism of $\Zq$-Banach modules:
\begin{align*}
\mathcal{C}\hat{\otimes}_{\Cp} \mathcal{H}_{r,s}&\cong( \Zq[[\pi]]\hat{\otimes}_{\Zq} \Cp)\hat{\otimes}_{\Cp}\mathcal{H}_{r,s}
	\cong\Zq[[\pi]]\hat{\otimes}_{\Zq} (\Cp\hat{\otimes}_{\Cp}\mathcal{H}_{r,s})\\
	&\cong\Zq[[\pi]]\hat{\otimes}_{\Zq}\mathcal{H}_{r,s}.
\end{align*}
So by Proposition A1.3 in \cite{Coleman}, $\{(1\otimes 1)\otimes e_i\}_{i\in I}$ is an orthonormal basis for $\mathcal{C}\hat{\otimes}_{\Cp} \mathcal{H}_{r,s}$ over $\mathcal{C}$, which implies that $\{1\otimes e_i\}_{i\in I}$ is an orthonormal basis for $\Zq[[\pi]]\hat{\otimes}_{\Zp}\mathcal{H}_{r,s}$ over $\mathcal{C}$. 
\end{proof}

Let  $\gls{Hrsj} = \Cp\langle \frac{R}{x-\widehat{P_j^s}}\rangle$ and define $\gls{Hrsjpi}=\Zq[[\pi]]{\hat{\otimes}}_{\Zq}(\mathcal{H}_{r,s})_j$. For each $j$, let $\gls{HrsjNorm}$ be the norm coming form the tensor product in $(\mathcal{H}_{r,s})_j$. 

\begin{prop}[Mittag-Leffler]\label{Mittag-Leffler}
There is a decomposition of $\Zq[[\pi]]$-Banach modules
	$$\mathcal{H}_{r,s}^\pi\cong \bigoplus_{j=1}^\ell (\mathcal{H}_{r,s}^\pi)_j.$$
Moreover, if for $\xi\in \mathcal{H}_{r,s}^\pi$ we write $\xi = \sum_{j=1}^\ell \xi_j\in \bigoplus_{j=1}^\ell (\mathcal{H}_{r,s}^\pi)_j$, then $\|\xi\|_{r,s}=\max_{1\leq j\leq \ell}\|\xi_j\|_j$.
\end{prop}
\begin{proof}
By Proposition 6 in section 2.1.7 of \cite{BGR}, 
\begin{align*}
	\mathcal{H}_{r,s}^\pi&\cong \Zq[[\pi]]{\hat{\otimes}}_{\Zq} \mathcal{H}_{r,s}\cong \Zq[[\pi]]{\hat{\otimes}}_{\Zq}\bigoplus_{j=1}^\ell (\mathcal{H}_{r,s})_j=\bigoplus_{j=1}^\ell\left ( \Zq[[\pi]]{\hat{\otimes}}_{\Zq}(\mathcal{H}_{r,s})_j \right ).
\end{align*}

The norm relationship follows from Proposition~\ref{structure_prop}. 
\end{proof}

%%%%%%%%%%%%%%%%%%%%%%%%%%%%%%%%%%%%%%%%%%
%
% The Submodule $\mathcal{Z}^\pi$
%
%%%%%%%%%%%%%%%%%%%%%%%%%%%%%%%%%%%%%%%%%%
\subsection{The Submodule $\mathcal{Z}^\pi$}

For the purposes of our Dwork theory, it will suffice to work in an integral submodule $\mathcal{Z}^\pi$ of $\mathcal{H}^\pi_{1,1}$. 

\begin{definition} Consider the $\Zp$ and $\Zq$-Banach modules: 
\begin{align*}
	\gls{O1} = \Zp[[\pi]]\hat{\otimes}_{\Zp}\Zp\textrm{ and }\gls{Oa} = \Zq[[\pi]]\hat{\otimes}_{\Zq}\Zq,
\end{align*} 
and define \gls{Z} to be the the submodule of $\mathcal{H}_{1,1}^\pi$ generated by tensors of the form $1\otimes B_{ij}$ with coefficients in $\mathcal{O}_a$.
\end{definition}

By Proposition~\ref{structure_prop}, every $\xi\in\mathcal{Z}^\pi\subset\mathcal{H}_{1,1}^\pi$ can be uniquely represented as a sum:
\begin{align}
\xi = \sum_{\substack{1\leq j\leq\ell\\ i\geq 0}}c_{ij}(1\otimes B_{ij}),
\end{align}\label{regular_ortho}
	with $c_{ij}\in\mathcal{O}_a$.
Or, via Proposition~\ref{padic_reich} and Proposition~\ref{structure_prop}, each $\xi\in\mathcal{Z}^\pi$ can be uniquely represented as
\begin{align} \label{reich_basis}
\xi = \sum_{\substack{1\leq j\leq\ell\\ i\geq 0}}e_{ij}(1\otimes \frac{x^i}{v^j}),
\end{align}
again with $e_{ij}\in\mathcal{O}_a$.

If $\Gal(\Qq/\Qp)=\langle \gls{tau}\rangle$, $\mathcal{O}_a$ can be endowed with a natural $\tau$ action,
	$$\tau(b(\pi)\otimes r)\mapsto \tau(b(\pi))\otimes\tau(r),$$
with the action of $\tau$ on $\Zq[[\pi]]$ defined coefficient-wise acting as the identity on $\pi$.  Furthermore, letting $\tau$ act as the identity on $x$, we get a $\tau$ action on $\mathcal{Z}^\pi$:
$$\sum_{\substack{1\leq j\leq\ell\\ i\geq 0}}c_{ij}(1\otimes \frac{1}{(x-\widehat{P}_j)^j})\mapsto \sum_{\substack{1\leq j\leq\ell\\ i\geq 0}}\tau(c_{ij})(1\otimes  \frac{1}{(x-\tau(\widehat{P}_j))^j}).$$

 (Note that this $\tau$ action is essentially the same action as $\tau_*$ from \cite{Zhu}.)

We also will need to define two handy maps associated to $\mathcal{Z}^\pi$. 
\begin{lemma}
There is an $\Zq$-Banach algebra isomorphism:
\begin{align*}
	\gls{iota}:\mathcal{O}_a&\to \Zq[[\pi]]\\
	b(\pi)\otimes r &\mapsto rb(\pi),
\end{align*}
and, for $x_0\in  \mathbb{A}_{1,1}$,  there is an evaluation map:
\begin{align*}
	\gls{rho}: \mathcal{Z}^\pi&\to \mathcal{C}\\
	\sum_{ij}c_{ij}(1 \otimes \left(\frac{1}{x-\widehat{P}_j}\right)^i ) &\mapsto 
		\sum_{ij}c_{ij}\left(\frac{1}{x_0-\widehat{P}_j}\right)^i.
\end{align*}
\end{lemma}
\begin{proof}
Defining the obviously bounded $\Zq$-algebra homomorphisms
 \begin{align*}
 \phi_1: \Zq[[\pi]]  \to\Zq[[\pi]]&:
  b(\pi)  \mapsto b(\pi)
  \\
  \phi_2:\Zq \to\Zq[[\pi]]&:
  a \mapsto a,
  \end{align*}
  by Proposition 2 in 3.1.1 of \cite{BGR}, there is a unique bounded $\Zq$-algebra homomorphism  $\psi: \Zq[[\pi]]\hat{\otimes}_{\Zq}\Zq\to\Zq[[\pi]]$. If $a(\pi)\otimes b\in \Zq[[\pi]]\hat{\otimes}_{\Zq}\Zq$,  it's easy to see that $a(\pi)\otimes b = ba(\pi)\otimes 1$, and so by the induced action of $\phi_1$ and $\phi_2$ through $\psi$, $\psi(a(\pi)\otimes b) = ba(\pi)$.  Hence if we define 
 \begin{align*}
 \psi': \Zq[[\pi]] &\to \Zq[[\pi]]\hat{\otimes}_{\Zq}\Zq\\
 	a(\pi) &\mapsto a(\pi)\otimes 1,
 \end{align*}
 one sees that $\psi\circ\psi'$ is the identity and thus $\psi$ is a bijection.

The only thing left is to check is that $\rho_{x_0}$ is well-defined. 
If $x_0\in  \mathbb{A}_{1,1}$, then $|x_0-\hat{P}_j|_p\geq 1$ and so $|\left(\frac{1}{x-\hat{P}_j}\right)^i|_p\leq 1$. Hence $|c_{ij}\left(\frac{1}{x-\hat{P}_j}\right)^i|_{\pi,p}\to 0$ as $i,j\to\infty$ since $|c_{ij}|_{\pi,p}\to 0$ as $i,j\to\infty$, and the claim follows. Observe that if $x_0\in\Zq$, then $\rho_{x_0}: \mathcal{Z}^\pi\to \mathcal{O}_a$ and $\iota\circ \rho_{x_0}: \mathcal{Z}^\pi\to\Zq[[\pi]]$.

\end{proof}

We will also need a twisting of $\mathcal{Z}^\pi$, \gls{Ztau}, which is defined to be the submodule of elements of the form
\begin{align}
\xi = \sum_{ij}c_{ij}(1\otimes \frac{1}{(x-\widehat{P}_j^p)^i}),
\end{align}\label{regular_ortho}
	with $c_{ij}\in\mathcal{O}_a$. We will write $\gls{Bijtau} = \frac{1}{(x-\widehat{P}_j^p)^i}$.

%%%%%%%%%%%%%%%%%%%%%%%%%%%%%%%%%%%%%%%%%%%%
%
% A Trace Formula 
%
%%%%%%%%%%%%%%%%%%%%%%%%%%%%%%%%%%%%%%%%%%%%
\section{A Trace Formula}

In this section we develop key trace formulas that will form the foundation for our corresponding Dwork theory. We will work towards proving the following theorem:

\begin{theorem}\label{trace_formula_gen}
Let $k\geq 1$ and $g\in\mathcal{Z}^\pi$ with $U^a\circ g$ completely continuous. Then 
	$$\Tr((U^a\circ g)^k  |\mathcal{Z}^\pi) =  (q^k-1)^{-1}\sum_{\substack{x_0\in\widehat{\mathbb{F}_{q^k}^\times},\\ x_0\neq\widehat{P}_1, \cdots, \widehat{P}_\ell}} \rho_{x_0}\circ (g(x)\cdots g(x^{q^{k-1}})),$$
\end{theorem}
where $U$ is defined below. 

\subsection{The $U_p$ Operator}

Let $U_p$ be the operator on $\mathcal{H}_{r,s}$ from \cite{Zhu}, namely:
\begin{align*}
	\gls{Up}: \mathcal{H}_{r,s} &\to \mathcal{H}_{r^p, sp}\\
		\xi(x) &\mapsto \frac{1}{p}\sum_{z^p=x}\xi(z).
\end{align*}
We can extend the $\Cp$-linear operator $U_p$ to a $\mathcal{C}$-linear operator on $\mathcal{H}_{r,s}^\pi$:
\begin{definition}  Let $U$ be the $\mathcal{C}$-linear operator given by
\begin{align*}
	\gls{U}: \mathcal{H}_{r,s}^\pi &\to \mathcal{H}_{r^p, sp}^\pi\\
		b\otimes \xi &\mapsto b\otimes U_p(\xi),
\end{align*}
and extended linearly. 
\end{definition}

\begin{prop}\label{U_props}
The operator $U_p$  has the following properties:
\begin{enumerate}
\item For $\xi$ and $g$, $U(\xi(x^q)g(x)) = \xi(x)U(g(x))$.
\item Let $h(x)=\sum_{i=-\infty}^\infty h_ix^i\in\Cp[[x, x^{-1}]]$. Then $U_ph = \sum_{i=-\infty}^\infty h_{pi}x^i$. 
\end{enumerate}
\end{prop}
\begin{proof}
The first result is trivial and the second is well known, see \cite{Robba}, p.238. 
\end{proof}

To prove the trace formula we'll need to understand exactly how $U$ acts on the $B_{ij}$:
\begin{lemma}\label{U_coeffs}
Let  $x\in\mathbb{A}_{r,1}$ and $B_{ij}^{\pi,\tau}=\frac{1}{(x-\hat{P}_j^p)^i}$. Then
$$UB_{ij}^\pi=\sum_{n=\floor{i/p}}^i (U_{(i,j), n}\otimes\hat{P}_j^{np-i}) B_{nj}^{\pi,\tau},$$
with $U_{(i,j), n}\in\Zp$. For $j=1,2$, $U_{(i,j), n}=0$ unless $n=i/p$, in which case $U_{(i,j),i/p} = 1$. When $j\geq 3$, $U_{(i,j), \ceil{i/p}}\in\Zptimes$ and $\ord U_{(ij),n)}\geq \frac{np-i}{p-1}-1$.
\end{lemma}
\begin{proof}
Apply Lemma 3.1 from \cite{Zhu}. See also section 5.3 in \cite{dwork_lectures}.
\end{proof}

Hence $U$ maps $\mathcal{Z}^\pi$ to $\mathcal{Z}^{\pi, \tau}$, implying that $U^a$ maps $\mathcal{Z}^\pi$ to $\mathcal{Z}^{\pi,\tau^a} = \mathcal{Z}^\pi$, i.e. $U^a$ is an endomorphism of $\mathcal{Z}^\pi$.
 
Let us finish this subsection by proving that $U^a$ is not only an endomorphism of $\mathcal{Z}^\pi$, but that it's a continuous endomorphism. 

\begin{prop}\label{Ug_norm}
Let $h\in \mathcal{Z}^\pi$. Then $U^a\circ h$ is a continuous linear operator, $h$ acting by multiplication, of norm $\leq q \|h\|_{r,s}$.
\end{prop}
\begin{proof}
We'll first prove that $U$ is a continuous linear operator of norm less than or equal to $p$. Unless noted, all of the following suprema are taken over $g\in\mathcal{H}_{r,s}^\pi,\\\|g\|_{r,s}=1$, and we write $g=\sum_{i,j}c_{ij}(1\otimes B_{ij})$. Because
\begin{align*}
\| U\|_{op}&=\sup\| U\circ g\|_{r,s}
	=\sup\|\sum_{ij}c_{ij}(1\otimes U_p\circ B_{ij}(x))\|_{r,s}\\
	&\leq\sup(\sup_{ij} \|c_{ij}\|_{\pi,p})(\sup_{ij} \|U_p\circ B_{ij}(x)\|_{r,s})
	=\sup\|g\|_{r,s}\|U_p\circ B_{ij}(x)\|_{op}\leq p,
\end{align*}
by Proposition~6 in \cite{Reich}, and so $U$ is continuous.

We conclude:
\begin{align*}
\| U^a\circ h\|_{op}&=\sup\| U^a(hg)\|\leq\sup\| U^a\|_{op}\|hg\|_{r,s}= q\|h\|_{r,s}.
\end{align*}
\end{proof}

\subsection{Building the Trace Formula}

This subsection contains the proof of our desired trace formula. The first step is to develop an  analogue trace formula on a polynomial submodule, $\mathcal{P}^\pi$.  Using a limiting process, we can then lift this formula to $\mathcal{Z}^\pi$, and this consequently yields Theorem~\ref{trace_formula_gen}.

\begin{definition}
Let $\gls{Ppi}$ be a submodule of $\mathcal{Z}^\pi$ spanned by tensors of the form $1\otimes x^i$, $i\geq 0$, over $\mathcal{O}_a$.  

For $g\in\mathcal{P}^\pi$ (or $\mathcal{Z}^\pi$), we say that $g$ is finite if it can be written as a finite sum:
	$$g= \sum_{j=1}^\ell \sum_{i=1}^{N_j} c_{ij}(1\otimes B_{ij}),$$
where $N_j<\infty$. 
\end{definition}

\begin{prop}\label{poly_trace}
Let $h\in\mathcal{P}^\pi$ and suppose that $U^a\circ h$ is completely continuous. Then
	$$\Tr(U^a\circ h | \mathcal{P}^\pi) = (q-1)^{-1}\sum_{x_0\in\widehat{\mathbb{F}_{q}^\times}} \rho_{x_0}\circ h.$$
\end{prop}

\begin{proof}

 Write $h=\sum_{i=0}^\infty c_{i}(1\otimes x^i)$, $c_i\in\mathcal{O}_a$. Applying Proposition~\ref{U_props} 
\begin{align*}
(U^a\circ h)(x) &= \sum_{i=0}^\infty c_i (1\otimes U(x^i)) =  \sum_{i=0}^\infty c_{qi}(1 \otimes x^i).
\end{align*}
Hence,
\begin{align*}
(U^{a}h)(1\otimes x^j)&=
	\sum_{i=0}^\infty c_{q^i} (1\otimes x^{i+j})
	=\sum_{i=0}^\infty c_{q^i-j}(1 \otimes x^{i}),
\end{align*}
and so
	$\Tr(U^a\circ h | \mathcal{P}^\pi) =\sum_{i=0}^\infty c_{(q-1)i}.$
The elementary fact that 
	$$\sum_{x_0\in\widehat{\mathbb{F}_{q}^\times}} x^w=\begin{cases}
        (q-1), & \text{if } (q-1)|w\\
      	0, & \text{if } (q-1)\nmid w
        \end{cases}$$
 yields the claim.
\end{proof}

Recall that in Reich's basis for $\mathcal{H}_{r,s}$ we used a polynomial $\gls{vx}=(x-\widehat{P}_1)\cdots (x-\widehat{P}_\ell)$. In what follows, we will need a lifting of $v$, $\gls{vpi}=1\otimes v$.

\begin{lemma}\label{technical_lemma}
For $x\in\mathbb{A}_{1,1}$, 
	$$|(v(x))^{(q-1)p^b}-(v(x^q)/v(x))^{p^b}|_p\leq p^{-(b+1)},$$
and consequently, $|(v^\pi(x))^{(q-1)p^b}-(v^\pi(x^q)/v^\pi(x))^{p^b}|_{r,s}\leq p^{-(b+1)}$.
\end{lemma}
\begin{proof}
See the proof of Theorem 4 in \cite{Reich}
\end{proof}

\begin{prop}\label{poly_trace_relation}
Let $g=\sum_{ij}c_{ij}(1\otimes B_{ij})\in \mathcal{Z}^\pi$ be finite and suppose that $U^a\circ g$ is completely continuous. Then
	$$\Tr(U^a\circ g |\mathcal{Z}^\pi) = \lim_{b\to\infty} \Tr(U^a\circ g(v^\pi)^{(q-1)p^b} | \mathcal{P}^\pi).$$
\end{prop}
\begin{proof}
	Take $b$ to be sufficiently large so that for every $j$, $g(v^\pi)^{(q-1)p^b}\in\mathcal{P}^\pi$ and note that $U^a(\mathcal{P}^\pi)\subseteq \mathcal{P}^\pi$. (Such a $b$ exists since $g$ is finite.) In other words, $U^a\circ g(v^\pi)^{(q-1)p^b}$ is an operator on $\mathcal{P}^\pi$, and we can write
	
\begin{align}\label{gamma_b}
U^a\circ g(v^\pi)^{(q-1)p^b}(1\otimes\frac{x^i}{v^j}) = \sum_{r,s} \gamma_{i,j,r,s}^{(b)} \otimes \frac{x^r}{v^s},
\end{align}
for some $\gamma_{i,j,r,s}^{(b)}\in \Zq$ and $r\geq 0$,  $(r, \ell)=1$ and $j\geq 0$. Similarly, $U^a\circ g$ is an operator on $\mathcal{Z}^\pi$, and so we expand it as
\begin{align}\label{gamma_no_b}
	U^a\circ g(1\otimes\frac{x^i}{v^j}) =\sum_{r,s} \gamma_{i,j,r,s} \otimes\frac{x^r}{v^s},
\end{align}
again some $\gamma_{i,j,r,s}\in\Zq$. 
	
Let $m$ be an integer such that $\frac{q \min_{ij}|c_{ij}|_{\pi,p}}{p^{b+1}}=p^{m-(b+1)}.$ Combining  Lemma~\ref{technical_lemma} and Proposition~\ref{Ug_norm} yields
	$$\|U^a\circ g\circ ((v^\pi(x))^{(q-1)p^b}-(v^\pi(x^q)/v^\pi(x))^{p^b})\|_{op}\leq p^{m-(b+1)}.$$
But 
\begin{align*}
	U^a\circ g\circ ((v^\pi(x))^{(q-1)p^b}&-(v^\pi(x^q)/v^\pi(x))^{p^b}) =\\& 
		U^a\circ g(v^\pi(x))^{(q-1)p^b}-(v^\pi(x))^{p^b}\circ U^a\circ g(v^\pi(x))^{-p^b},
\end{align*}
and multiplying by $(v^\pi(x))^{-p^b}$ yields 
\begin{align}\label{c1}
\|(v^\pi(x))^{-p^b}\circ (U^a\circ g(v^\pi(x))^{(q-1)p^b})-U^a\circ g(v^\pi(x))^{-p^b}\|_{op}\leq p^{m-(b+1)}.
\end{align}
Substituting the expansions in (\ref{gamma_b}) and (\ref{gamma_no_b}) into (\ref{c1}) yields

\begin{align}\label{c4}
\|\sum_{r,s} \gamma_{i,j,r,s}^{(b)} &\otimes\frac{x^r}{v^{s+p^b}}-\sum_{r,s} \gamma_{i,j-p^b,r,s} \otimes\frac{x^r}{v^s}\|_{r,s}\leq p^{m-(b+1)}.
\end{align}
By definition then, (\ref{c4})  implies
\begin{align}
|\gamma_{ij,ij}^{(b)}-\gamma_{i,j-p^b,i,j-p^b}|_p\leq p^{m-(b+1)},
\end{align}
and so
\begin{align}
|\sum_{i,j}\gamma_{ij,ij}^{(b)}-\sum_{i\geq 0, j\geq p^b}\gamma_{ij,ij}|_p\leq p^{m-(b+1)}.
\end{align}
As $b\to\infty$ then, the identity follows. 
\end{proof}

\begin{theorem}\label{trace_formula}
Let $k\geq 1$, $g\in\mathcal{Z}^\pi$ and suppose that $U^a\circ g$ is completely continuous. Then 
	$$\Tr(U^a\circ g |\mathcal{Z}^\pi) =  (q-1)^{-1}\sum_{\substack{x_0\in\widehat{\mathbb{F}_{q}^\times},\\ x_0\neq\widehat{P}_1, \cdots, \widehat{P}_\ell}} \rho_{x_0}\circ g.$$
\end{theorem}
\begin{proof}
First suppose that $g$ is finite. Applying Proposition~\ref{poly_trace_relation} and Proposition~\ref{poly_trace} yields:
\begin{align*}
\Tr(U^a\circ g | \mathcal{Z}^\pi) &=\lim_{b\to\infty} \Tr(U^a\circ g(v^\pi)^{(q-1)p^b} | \mathcal{P}^\pi) \\
	&= \lim_{b\to\infty} (q-1)^{-1}\sum_{x_0\in\widehat{\mathbb{F}_{q}^\times}} \rho_{x_0}\circ (g(v^\pi)^{(q-1)p^b}).
\end{align*}
Now, if $x_0=\widehat{P}_j$ for any $j$, then for large $b$ it is clear that $\rho_{x_0}\circ (g(v^\pi)^{(q-1)p^b})=0$. On the other hand, if $x_0\neq\widehat{P}_j$ for all $j$, observe that since $x_0$ and $\widehat{P}_j$ are Teichmuller lifts, $|x_0|_p=|\widehat{P}_j|_p=1$. By assumption $\widehat{x_0}\neq P_j\in\Fq$, so $|x_0-\widehat{P}_j|_p\not < 1$ and $|x_0-\widehat{P}_j|_p=1$ and  $x_0-\widehat{P}_j\in\Zqtimes$. Therefore, by the discussion on p.150 in \cite{Robert}, $\lim_{b\to\infty} (x_0-\widehat{P}_j)^{(q-1)p^b}=1$, which implies that 
	$$\lim_{b\to\infty}\rho_{x_0}\circ (g(v^\pi)^{(q-1)p^b}) = \rho_{x_0}\circ g.$$ 
Consequently,
\begin{align*}
\Tr(U^a\circ g |\mathcal{Z}^\pi) &= (q-1)^{-1}\sum_{\substack{x_0\in\widehat{\mathbb{F}_{q}^\times},\\ x_0\neq\widehat{P}_1, \cdots, \widehat{P}_\ell}} \rho_{x_0}\circ g.
\end{align*}
The result for arbitrary $g$ then follows by taking limits. 
\end{proof}

The proof of Theorem~\ref{trace_formula_gen} follows similarly. (Apply property (1) from Proposition~\ref{U_props} to $(U^a\circ g)^k$ and replace $a$ with $ak$ in the above proofs.) 

%%%%%%%%%%%%%%%%%%%%%%%%%%%%%%%%%%%%%%%%%%%%
%
% $(\pi,p)$-adic Exponential Sums
%
%%%%%%%%%%%%%%%%%%%%%%%%%%%%%%%%%%%%%%%%%%%%
\section{$(\pi,p)$-adic Exponential Sums}

In this section we apply the above analysis to $(\pi,p)$-adic exponential sums.  
We describe $C_f(s,\pi)$ as the determinant of a completely continuous operator and compute estimates that will be  fundamental to the computation of the Hodge polygon in Section 7. 

Recall that $\gls{E} = \sum_{k=0}^\infty u^kx^k\in\Zp [[x]]$ is the Artin-Hasse exponential function and $\gls{pi}\in 1+ \Qp[[x]]$ is such that $E(\pi)=1+T$. Let $f(x)=\sum_{j=1}^\ell\sum_{i=1}^{d_j}\gls{aij}\left(\frac{1}{x-\widehat{P_j^s}}\right)^i$, $a_{ij}\in\Zq$,  and define its associated data:

\begin{definition}
\begin{align*}
\gls{Sf} &= \sum_{\substack{x\in\widehat{\mathbb{F}_{q^k}^\times},\\ x\neq\widehat{P}_1, \cdots, \widehat{P}_\ell}}E(\pi)^{\Tr_{\mathbb{Q}_{q^k}/\Qp}(f(x))}\\
\gls{Lf} &= \exp(\sum_{k=1}^\infty S_f(k, \pi)\frac{s^k}{k})\\
\gls{Cf} &=\exp(\sum_{k=1}^\infty -(q^k-1)^{-1} S_f(k, \pi)\frac{s^k}{k})=\prod_{j=0}^\infty L_f(q^js, \pi).
\end{align*}
\end{definition}

The function $f$ has the the splitting functions:
\begin{definition}
\begin{align*}
\gls{Fj} &= \prod_{i=1}^{d_j} E(\pi a_{ij}\otimes B_{ij})\\
\gls{F} &= \prod_{j=1}^\ell F_j(x)\\
\gls{Fa} &= \prod_{m=0}^{a-1}(\tau^m F)(x^{p^m}).
\end{align*}
\end{definition}

Our main object of study will be the maps $\gls{alphaa} = U^a\circ F_{[a]}$ and $\gls{alpha1}=\tau_{-1}\circ U\circ F$. Note that $\alpha_1$ is a $\mathcal{O}_1$-linear endomorphism of $\mathcal{Z}^\pi$ while $\alpha_a$ is a $\mathcal{O}_a$-linear endomorphism of $\mathcal{Z}^\pi$. They are related in the following manner:

\begin{prop}\label{alpha_relation}
As $\mathcal{O}_1$-linear maps, $\alpha_a=\alpha_1^a$ and $\det_{\mathcal{O}_a}(1-\alpha_as)^a=\det_{\mathcal{O}_1}(1-\alpha_1s)$.
\end{prop}
\begin{proof}
The proof of this proposition is similar the proof of Lemma 2.9 in \cite{Zhu} (or originally (43) in \cite{Bombieri}.)
\end{proof}

\subsection{$(\pi,p)$-adic Estimates}

The following are $(\pi,p)$-adic liftings of the $p$-adic approximations from \cite{Zhu}. Lemma~\ref{Fij_bound} and Lemma~\ref{F_ijnk_bound} are purely $\pi$-adic estimates, and the key computation, Proposition~\ref{unweighted_bound},  blends these two $\pi$-adic estimates with the $p$-adic nature of the $U$ operator, Lemma~\ref{U_coeffs}.
 
For the sake of notation, we will write our (unweighted) basis as $B_{ij}^\pi = 1\otimes B_{ij}$ (similarly  $B_{ij}^{\pi, \tau} = 1\otimes B_{ij}^{\tau}$) and define a weighted basis $W_{ij}^\pi=\pi^{\frac{i}{d_j}}\otimes B_{ij}$.

\begin{definition}
Let $i\geq 0$ and $0\leq j, k\leq \ell$ and define
\begin{alignat*}{2}
U(B_{ij}^\pi) &= \sum_{i,j} \gls{Uijn}B_{nj}^{\pi, \tau}, &&U_{(i,j),n}\in\Zp \\
F_j(x) &=\sum_{n=0}^\infty \gls{Fnj}\otimes B_{nj}, &&F_{n,j}\in\Zq[[\pi]] \\
(FB_{ij}^\pi)_k &= \sum_{n=0}^\infty \gls{Fijnk}\otimes B_{nk},\ &&F_{(i,j),(n,k)}\in\Zq[[\pi]].
\end{alignat*}
\end{definition}

\begin{lemma}\label{Fij_bound}
The coefficient $F(x)\in\mathcal{Z}^\pi$ and $\ordp_{\pi}F_{nj} \geq \ceil{\frac{n}{d_j}}$ for each $j$. Moreover, if $d_j|n$, equality holds.
\end{lemma}
\begin{proof}
By definition,
\begin{align*}
F_j(x) &= \prod_{i=0}^{d_j}\left (\sum_{k=0}^\infty u_ka_{ij}^k\pi^k\otimes B_{ij}^k\right) = 
	\sum_{n=0}^\infty \left (\sum_{\substack{\sum_{k=1}^{d_j} kn_k=n\\n_k\geq 0}}\prod_{k=1}^{d_j}u_{n_k}a_{kj}^{n_k}\pi^{n_k}\right)\otimes B_{nj},
\end{align*}
and so
\begin{align*}
F_{n,j} &= \sum_{\substack{\sum_{k=1}^{d_j}kn_k=n\\n_k\geq 0}}\left (\prod_{k=1}^{d_j}u_{n_k}a_{kj}^{n_k}\right )\pi^{\sum_{k=1}^{d_j}n_k}.
\end{align*}
Taking $n_{d_j}=\floor{\frac{n}{d_j}}$ and $n_{n\bmod d_j}$ to be either $0$ or $1$ depending on if $n\bmod d_j=0$ or $n\bmod d_j\neq 0$ respectively yields the claim.  When $d_j| n$, equality follows from the fact that both $a_{d_j, j}$ and $u_{\frac{n}{d_j}}$ are nonzero. (The Artin-Hasse coefficient $u_n$ can be expressed as $u_n=h_n/n!$, where $h_n$ is the number of $p$-elements in the permutation group $S_n$. The fact that $u_n\neq 0$ is then immediate.)
\end{proof}

\begin{lemma}\label{F_ijnk_bound}
Fix $i,n\geq 0$ and $1\leq j,k\leq \ell$. Then:
\begin{align*}
\ordp_{\pi} F_{(ij),(nk)}\geq\begin{cases}
	\frac{n-i}{d_k} &\textrm{ if } j=k\\
	\frac{n+i}{d_k} &\textrm{ if } j\neq 1, k=1\\
	\frac{n}{d_k} &\textrm{ if } j\neq k, k\neq 1,
\end{cases}
\end{align*}
and equality holds when $d_k|(n-i)$, $d_k|(n+i)$ or $d_k|n$ respectively. 
\end{lemma}
\begin{proof}
First, observe
\begin{align}\label{big_exp}
FB_{ij}^\pi &= \left (\sum_{m=0}^\infty F_{m,j}\otimes B_{m+i, j}\right )\prod_{\substack{v=1\\v\neq j}}^\ell\left (\sum_{m=0}^\infty F_{m,v}\otimes B_{m,v}\right), 
\end{align}
where the only $\pi$-adic terms come from the $F_{m,k}$ and $F_{m,v}$ terms. If we want to compute $(FB_{ij}^\pi)_k$, we need to expand each $B_{m,v}$, $v\neq k$, in terms of $\frac{1}{x-\widehat{P_k}}$. There are several cases to consider: 

If $v\geq 2$ and $k\geq 3$, $v\neq k$, to expand $\frac{1}{x-\widehat{P_v}}$ in terms of $\frac{1}{x-\widehat{P_k}}$:
\begin{align}\label{expansion}
\frac{1}{x-\widehat{P_v}} = \frac{1}{\widehat{P_k}-\widehat{P_v}}\frac{1}{1-(\frac{x-\widehat{P_k}}{\widehat{P_v}-\widehat{P_k}})}=
	\sum_{m=0}^\infty(-1)^m(\widehat{P_v}-\widehat{P_k})^{-(m+1)}(x-\widehat{P_k})^m,
\end{align}
which is analytic on the ball with $|x-\widehat{P_k}|_p<|\widehat{P_v}-\widehat{P_k}|_p=1$. 

If $v\geq 3$ and $k=1$, use
\begin{align}\label{expansion_2}
\frac{1}{x-\widehat{P_v}} =\frac{1}{x}\cdot\frac{1}{1-\frac{\widehat{P_v}}{x}}=
	\sum_{m=0}^\infty\frac{\widehat{P_v^m}}{x^{m+1}},
\end{align}
which converges on $|x|_p>1$.

If $v\geq 3$ and $k=2$, use
\begin{align}\label{expansion_3}
\frac{1}{x-\widehat{P_v}} =-\frac{1}{\widehat{P_v}}\cdot\frac{1}{1-\frac{x}{\widehat{P_v}}}=
	-\frac{1}{\widehat{P_v}}\sum_{m=0}^\infty\frac{x^m}{\widehat{P_v^m}},
\end{align}
which converges on $|x|_p<|\widehat{P_v}|_p=1$.
If $v=1$ and $k\geq 3$, just use the trivial expansion $x = (x-\widehat{P_k})+\widehat{P_k}$. Finally, if $v=1$ and $k=2$ (or vice versa), no expansion is necessary. 

Let's start with the case $j=k=1$:
\begin{align}\label{j=k=1}
FB_{i,1}^\pi = \left (\sum_{m=0}^\infty F_{m,1}\otimes x^{m+i}\right ) &\left(\sum_{m=0}^\infty F_{m,2}\otimes \frac{1}{x^m}\right)\cdot \\
	&\prod_{v=3}^\ell\left (\sum_{m=0}^\infty F_{m,v}\otimes \left (\sum_{w=0}^\infty\frac{\widehat{P_v^w}}{x^{w+1}}\right)^m\right). \nonumber
\end{align}
Since we only care about the $\pi$-terms, it's clear that the minimum occurs from the term $F_{n-i,1}\otimes x^n$, and the bound follows from Lemma~\ref{Fij_bound}. The case for $j=k=2$ is similar.

Now, let's look at the case $j=k\geq 3$. For each $v\neq j$, expand $B_{m,v}$ as above. Then $F_{(ij),(nk)}$ is the coefficient of $B_{nk}^\pi$ in (\ref{big_exp}) after substituting all appropriate expansions. Each expansion has only positive powers of $(x-\widehat{P_k})$, and so 
\begin{align}\label{F_ijnk_inequal}
\ordp_{\pi} F_{(ij),(nk)}\geq \min_{(n_1,\cdots, n_\ell)}\ordp_{\pi}\prod_{\substack{v=1}}^\ell F_{v, n_v},
\end{align}
	where the minimum is over all $(n_1,\cdots, n_\ell)\in\mathbb{Z}_{\geq 0}^\ell$ such that $n_k-\sum_{\substack{v=1\\v\neq k}}^\ell n_v = n-i$. Clearly this occurs when $n_k=n-i$ and $n_v=0$ for $v\neq k$. The bound follows after applying Lemma~\ref{Fij_bound} to (\ref{F_ijnk_inequal}).

In the case $j\neq 1$, $k=1$, if $j=2$,
\begin{align*}
FB_{i,1}^\pi = \left (\sum_{m=0}^\infty F_{m,1}\otimes x^{m}\right ) &\left(\sum_{m=0}^\infty F_{m,2}\otimes \frac{1}{x^m}\right)\cdot \\
	&\prod_{v=3}^\ell\left (\sum_{m=0}^\infty F_{m,v}\otimes \left (\sum_{w=0}^\infty\frac{\widehat{P_v^w}}{x^{w+1}}\right)^m\right)\cdot \frac{1}{x^i},
\end{align*}
and so again the term contributing to the coefficient of $B_{nk}$ giving smallest $\pi$-adic term is $F_{n+i}\otimes x^{n+i}$. The case $j\geq 3$ is similar. 

Finally, there's the case $j\neq k$, $k\neq 1$. Suppose that $j,k\geq 3$. (The other cases are again similar.) Then the expansion of each $B_{m,v}$ in terms of $k$, including the $B_{ij}$ have only positive powers of $(x-\widehat{P_k})$ and so the minimum occurs simply at $F_{n,k}\otimes x^n$. 

Note that in all of the above estimates, if $d_k|(n-i)$, then by Lemma~\ref{Fij_bound} the minimum obtained in (\ref{F_ijnk_inequal}) is unique and sharp and equality holds. 
\end{proof}

\begin{definition} \label{map_coeffs}
Fix $i,n\geq 0$ and $1\leq j,k\leq \ell$ and recall that for $\xi\in\mathcal{H}_{r,s}^\pi$, $(\xi)_j$ is the Laurent expansion at $\hat{P}_j$. We write:
\begin{alignat*}{2}
(\alpha_1 B_{ij}^\pi)_{k} &= \sum_{n=0}^\infty \gls{Cijnk}\otimes B_{n,k}^{\pi},\ && C_{(i,j),(n,k)}\in\Zq[[\pi]]  \\
(\alpha_1 W_{ij}^\pi)_{k} &= \sum_{n=0}^\infty \gls{Dijnk}\otimes W_{n,k}^{\pi},\ && D_{(i,j),(n,k)}\in\Zq[[\pi]].
\end{alignat*}
\end{definition}

\begin{prop}\label{unweighted_bound}
Fix $i,n\geq 0$ and $1\leq j,k\leq \ell$. Then if $k=1,2$:
\begin{align*}
\ordp_{\pi} C_{(ij),(nk)}\geq \frac{pn-i}{d_k}.
\end{align*}

For $k\geq 3$,
\begin{align*}
 \ordp_{\pi} C_{(ij),(nk)} \geq\begin{cases}
	\frac{n-i}{d_k} &\textrm{ if } j=k\\
	\frac{n+i}{d_k} &\textrm{ if } j\neq 1, k=1\\
	\frac{n}{d_k} &\textrm{ if } j\neq k, k\neq 1,
\end{cases}
\end{align*}
and equality holds when $d_k|(n-i)$, $d_k|(n+i)$ or $d_k|n$ respectively. For $k\geq 3$ and any real number $c>0$, $C_{(ij),(nk)}$ also has the following $(\pi^{1/c},p)$-adic estimates:
\begin{align*}
 \ordp_{\pi^{1/c}, p} C_{(ij),(nk)}\geq\begin{cases}
	\frac{(n-1)p-(i-1)}{d_k}c &\textrm{ if } d_k\geq c(p-1)\\
	\frac{n-i}{d_k}c+ n-1 &\textrm{ if } d_k<c(p-1).
\end{cases}
\end{align*}
\end{prop}
\begin{proof}
We'll prove the $(\pi,p)$-adic bound, and the $\pi$-adic bounds follow easily. Let $B_{nk}^{\pi,\tau}=\tau(B_{ij}^\pi)$. Then,
\begin{align*}
\tau\circ\alpha_1B_{ij}^\pi &=(U\circ F)B_{ij}^\pi = U\circ(\sum_{k=1}^\ell(FB_{ij}^\pi)_k) = \sum_{k=1}^\ell\sum_{n=0}^\infty F_{(ij),(nk)}\otimes U(B_{n,k}^\pi)\\
	&=\sum_{k=1}^\ell\sum_{n=0}^\infty F_{(ij),(nk)}\sum_{m=\ceil{n/p}}^n(U_{(n,k),m}\otimes \widehat{P_k}^{mp-n})B_{mk}^{\pi,\tau}\\
	&=\sum_{k=1}^\ell\sum_{m=0}^\infty\left [ \sum_{n=m}^{mp} F_{(ij),(nk)}(U_{(n,k),m}\otimes \widehat{P_k}^{mp-n})\right]B_{mk}^{\pi,\tau},
\end{align*}
and so 
\begin{align}\label{AA1}
 C_{(ij),(mk)} = \tau^{-1}\circ\sum_{n=m}^{mp} F_{(ij),(nk)}(U_{(n,k),m}\widehat{P_k}^{mp-n}).
 \end{align}

For $k=1$ and $2$, Proposition~\ref{U_props} implies that $U_{(n,k),m}=0$ for $m\neq np$, and combined with Lemma~\ref{F_ijnk_bound}, this yields the first part of the claim. 

For $k\geq 3$, by (\ref{AA1}),
\begin{align}\label{weight_B_basic}
\ordp_{\pi^{1/c}, p}C_{(ij),(mk)}\geq& \min_{m\leq n\leq mp}(\ordp_{\pi^{1/c}}F_{(ij),(nk)}+\ord U_{(nk),m}).
\end{align}
Let $n_0=(m-1)p+1$. By Lemma~\ref{U_coeffs}, if $n_0<n\leq mp$, $\ord U_{(nk),m}=0$ and 
so (\ref{weight_B_basic}) yields $\ordp_{\pi^{1/c}} F_{(ij),(nk)}\geq \frac{n_0-i}{d_k}c$. On the other hand, if $m\leq n\leq n_0$,
\begin{align}\label{AA2}
\ordp_{\pi^{1/c}, p}C_{(ij),(mk)}\geq&  \min_{m\leq n\leq n_0}(\frac{n-i}{d_k}c + \frac{mp-n}{p-1}-1)\nonumber\\
	\geq& \min_{m\leq n\leq n_0}((\frac{-ic}{d_k}+\frac{mp}{p-1}-1)+ n(\frac{c}{d_k} -\frac{1}{p-1})).
\end{align}

There are now three cases to consider. First, if $\frac{c}{d_k} -\frac{1}{p-1}<0$, then (\ref{AA2}) has minimum at $n=n_0=(m-1)p+1$, which yields
	$\ordp_{\pi^{1/c}, p}C_{(ij),(mk)}\geq \frac{(m-1)p-(i-1)}{d_k}c$.
If $\frac{c}{d_k}-\frac{1}{p-1}\geq 0$, then (\ref{AA2}) has minimum at $n=m$, and lower bound $\frac{m-i}{d_k}c+(m-1)$.
\end{proof} 

\begin{theorem}\label{weighted_B_bound}
Fix $i,n\geq 0$ and $1\leq j,k\leq \ell$. Using the relation $D_{(ij),(nk)}=\pi^{i/d_j-n/d_k} C_{(ij),(nk)}$ and Proposition~\ref{unweighted_bound}, if $k=1,2$:
\begin{align*}
\ordp_{\pi} D_{(ij),(nk)}\geq \frac{(p-1)n}{d_k}.
\end{align*}

For $k\geq 3$,
\begin{align*}
 \ordp_{\pi} D_{(ij),(nk)} \geq\ 0
\end{align*}
and equality holds when $d_k|(n-i)$ and $j=k$. Furthermore, for a real number $c>0$, 
\begin{align*}
 \ordp_{\pi^{1/c}, p} D_{(ij),(nk)}\geq\begin{cases}
	\frac{(n-1)(p-1)}{d_k}c &\textrm{ if }d_k\geq c(p-1)\\
	n-1 &\textrm{ if } d_k<c(p-1).
\end{cases}
\end{align*}
\end{theorem}

\begin{corollary} \label{not_cc}
Neither $\alpha_1$ nor $\alpha_a$ are $\pi$-adically completely continuous operators, but for $c>0$, they are both $(\pi^{1/c},p)$-adically completely continuous operators.
\end{corollary} 
\begin{proof}
To see that $\alpha_1$ is not completely continuous $\pi$-adically, see by Theorem~\ref{weighted_B_bound} that if $j=k$ and $d_k|(n-i)$, then $\ordp_{\pi}D_{(ij),(nk)} = 0$. Hence
	$$\lim_{(n,k)\to\infty}\inf_{(i,j)} \ordp_{\pi}D_{(i,j),(n,k)}=0,$$
and so $\alpha_1$ cannot be completely continuous with respect to $\pi$.
 
On the other hand, the $(\pi,p)$-adic bound from Theorem~\ref{weighted_B_bound} (without loss of generality, take $k\geq 3$ and $ d_k>p-1$) implies that 
	$$\ordp_{\pi^{1/c},p}D_{(i,j),(n,k)}\geq \frac{(p-1)(n-1)}{d_{k}}c\to\infty\textrm{ as } n\to\infty.$$
The complete continuity of $\alpha_a$ follows from the relation $\alpha_a=\alpha_1^a$. 
\end{proof}

\subsection{Dwork Theory}

\begin{lemma}\label{Tadic_relation}
Let $x_0\in\widehat{\mathbb{F}_{q^k}^\times}$ such that $x_0\neq \widehat{P_j}$ for all $1\leq j\leq \ell$. Then:
	$$\iota\circ\rho_{x_0}\circ \prod_{i=0}^{k-1}F_{[a]}(x_0^{q^i})= (1+T)^{\Tr_{\mathbb{Q}_{q^k}/\Qp}(f(x_0))}.$$
\end{lemma}
\begin{proof}
Let $x_0\in\widehat{\mathbb{F}_{q^k}^\times}$ with $x_0\neq \widehat{P}_j$  for all $1\leq j\leq \ell$. An easy calculation shows that
\begin{align*}
(1+T)^{\Tr_{\mathbb{Q}_{q^k}/\Qp}(f(x_0))} &= E(\pi)^{\sum_{j=1}^\ell\sum_{i=1}^{d_j}\sum_{m=0}^{ak-1}(a_{ij}(x_0-\widehat{P_j})^{-i})^{\tau^m}}\\
	&= \prod_{j=1}^\ell\prod_{i=1}^{d_j}\prod_{m=0}^{ak-1}E(\pi(a_{ij}^{\tau^m}(x_0^{p^m}-\widehat{P_j}^{p^m})^{-i})).
\end{align*}
On the other hand, 
\begin{align*}
 \prod_{i=0}^{k-1}F_{[a]}(x_0^{q^i})=\prod_{m=0}^{ak-1}(\tau^m F)(x) &=\prod_{j=1}^\ell\prod_{i=1}^{d_j}\prod_{m=0}^{ak-1}E(\pi(a_{ij}^{\tau^m}\otimes(x_0-\widehat{P_j}^{p^m})^{-i})),
\end{align*}
and the identity follows. 
\end{proof}

\begin{prop}\label{trace_explicit}
For $k\geq 1$,
	$$\iota\circ\Tr(\alpha_a^k |\mathcal{Z}^\pi) = (q-1)^{-1}S_f(k,\pi).$$
\end{prop}
\begin{proof}
Applying Theorem~\ref{trace_formula_gen} to the function $F_{[a]}(x)=\prod_{m=0}^{ak-1}(\tau^m F)(x^{p^m})$ and using the identity from Lemma~\ref{Tadic_relation} yields:
\begin{align*}
\iota\circ \Tr(\alpha_a^k |\mathcal{Z}^\pi) &=  (q^k-1)^{-1}\sum_{\substack{x_0\in\widehat{\mathbb{F}_{q^k}^\times},\\ x_0\neq\widehat{P}_1, \cdots, \widehat{P}_\ell}} \rho_{x_0}\circ (F_{[a]}(x)\cdots F_{[a]}(x^{q^{k-1}}))\\
	&=  (q^k-1)^{-1}\sum_{\substack{x_0\in\widehat{\mathbb{F}_{q^k}^\times},\\ x_0\neq\widehat{P}_1, \cdots, \widehat{P}_\ell}} (1+T)^{\Tr_{\mathbb{Q}_{q^k}/\Qp}(f(x_0))}\\
	&= (q^k-1)^{-1}S_f(k,\pi).
\end{align*}
\end{proof}

\begin{theorem} We have
$$C_f(s,\pi)=\iota\circ\det(1- \alpha_a s | \mathcal{Z}^\pi).$$
\end{theorem}
\begin{proof}
By definition and the trace formula in Corollary~\ref{trace_explicit}:
\begin{align*}
	C_f(s,\pi) &= \exp(-\sum_{k=1}^\infty(q^k-1)^{-1}S_f(k,\pi)\frac{s^k}{k})\\
		&=  \exp(-\sum_{k=1}^\infty\iota\circ\Tr(\alpha_a^k |\mathcal{Z}^\pi) \frac{s^k}{k})\\
		&=\iota\circ\det(1- \alpha_a s | \mathcal{Z}^\pi).
\end{align*}
\end{proof}

%%%%%%%%%%%%%%%%%%%%%%%%%%%%%%%%%%%%%%%%%%%%
%
% The Hodge Bound
%
%%%%%%%%%%%%%%%%%%%%%%%%%%%%%%%%%%%%%%%%%%%%
\section{The Hodge Bound}

We call the lower bound for $C_f(s, \pi)$ obtained from Theorem~\ref{weighted_B_bound} the Hodge bound. For two Newton polygons $\mathrm{NP}_1$, and $\mathrm{NP}_2$, let $\mathrm{NP}_1 \boxplus\mathrm{NP}_2$ denote  the concatenation of the Newton polygons $\mathrm{NP}_1$ and $\mathrm{NP}_2$, reordering so that the slopes are in increasing order. The Hodge polygon is then given by:

\begin{definition}
For $k=1,2$, let $\gls{HPkc}$ be the Hodge polygon with vertices 
\[
	\{(n, \frac{(p-1)n(n-1)}{2d_k}c)\}_{n\geq 0}.
\]
For $3\leq k\leq \ell$, let $\mathrm{HP}_k^c$ be the Hodge polygon with vertices $\{(n, y_n)\}_{n\geq 0}$, where

\begin{align*}
y_n=\begin{cases}
	\frac{a(p-1)n(n-1)}{2d_k}c &\textrm{ if } d_k\geq c(p-1)\\
	\frac{an(n-1)}{2} &\textrm{ if } d_k<c(p-1).
\end{cases}
\end{align*} The $(\pi, p)$-adic Hodge polygon, $\gls{HPc}$, is the polygon given by $\boxplus_{k=1}^\ell \mathrm{HP}_k^c$. 
\end{definition}

\begin{theorem}
The $(\pi^{1/c}, p)$-adic Newton polygon of $C_f(s,\pi)$ lies above $\mathrm{HP}^c$.
\end{theorem}
\begin{proof}
Let $M$ represent the matrix for $\alpha_1$ with respect to the basis $\{W_{ij}^\pi\}_{ij}$, with the entries of $M$ lying in $\mathcal{O}_a$. Write:
$$\det(1-Ms) = 1+\sum_{k=1}^\infty C_ks^k\in\mathcal{O}_a[[s]],$$
so that
\begin{align}\label{Ck_formula}
C_k = \sum_{\substack{S\subseteq\mathbb{Z}_{\geq 0}\times\{1,\cdots,\ell\}\\ |S|=k}}\sum_{\sigma\in\Sym(S)}\sgn\sigma\prod_{(i,j)\in S} D_{(i,j),\sigma(i,j)}.
\end{align}
	
\sloppy Let $m_i$ be the $i$th slope of $\mathrm{HP}^c$. The smallest $(\pi, p)$-adic valuation that 
$\prod_{(i,j)\in S} D_{(i,j),\sigma(i,j)}$ can have is $\sum_{i=1}^k m_i$, by Theorem~\ref{weighted_B_bound}, and so the desired Hodge bound holds for $\det_{\mathcal{O}_a}(1-\alpha_1s)$. 

However we need to show the Hodge bound holds for $\det_{\mathcal{O}_1}(1-\alpha_1s)$, so let $\eta, \cdots, \eta^{\tau^{a-1}}$ be a normal basis for $\Zq/\Zp$. Consider the $\mathcal{O}_1$-basis $\eta^{\tau^i}\otimes 1$, $0\leq i\leq \ell$, for $\mathcal{O}^a$. Because $\alpha_1$ is $\tau^{-1}$-linear, 
\[
	\alpha_1((\eta^{\tau^i}\otimes 1)\cdot C_{(i',j'),(nk)})=(\eta^{\tau^{i-1}}\otimes 1)\cdot \alpha_1(C_{(i',j'),(nk)}),
\]
 and so the bound follows from Proposition~\ref{alpha_relation} and Theorem~\ref{weighted_B_bound}.
\end{proof}

\printglossaries

\end{document}

%% file: glossary.tex
%%%%%%%
% Notation
%%%%%%%

\newglossaryentry{TateAlgebra}
{
	type = notation,
 	name = {\ensuremath{A\langle X_1, \cdots, X_n\rangle}},
 	sort = A,
 	description = {}
}

\newglossaryentry{GaussNorm}
{
	type = notation,
 	name = {\ensuremath{|\cdot |_{gauss}}},
 	sort = Norm,
 	description = {}
}

\newglossaryentry{OperatorNorm}
{
	type = notation,
 	name = {\ensuremath{\|\cdot\|_{op}}},
 	description = {},
 	sort = Norm
}

\newglossaryentry{Hrs}
{
	type = notation,
 	name = {\ensuremath{\mathcal{H}_{r,s}}},
 	description = {},
 	sort = H
}

\newglossaryentry{HrsNorm}
{
	type = notation,
 	name = {\ensuremath{\|\cdot\|_{r,s}}},
 	description = {},
 	sort = Norm
}

\newglossaryentry{pipNorm}
{
	type = notation,
 	name = {\ensuremath{|\cdot|_{\pi,p}}},
 	description = {},
 	sort = Norm
}

\newglossaryentry{Hrspi}
{
	type = notation,
 	name = {\ensuremath{\mathcal{H}_{r,s}^\pi}},
 	description = {},
 	sort = Hpi
}

\newglossaryentry{C}
{
	type = notation,
 	name = {\ensuremath{\mathcal{C}}},
 	description = {},
 	sort = C
}

\newglossaryentry{Hrsj}
{
	type = notation,
 	name = {\ensuremath{(\mathcal{H}_{r,s})_j}},
 	description = {},
 	sort = Hrsj
}

\newglossaryentry{Hrsjpi}
{
	type = notation,
 	name = {\ensuremath{(\mathcal{H}_{r,s}^\pi)_j}},
 	description = {},
 	sort = Hrsjpi
}

\newglossaryentry{HrsjNorm}
{
	type = notation,
 	name = {\ensuremath{\|\cdot\|_j}},
 	description = {},
 	sort = Norm
}

\newglossaryentry{Oa}
{
	type = notation,
 	name = {\ensuremath{\mathcal{O}_a}},
 	description = {},
 	sort = Oa
}

\newglossaryentry{O1}
{
	type = notation,
 	name = {\ensuremath{\mathcal{O}_1}},
 	description = {},
 	sort = O1
}

\newglossaryentry{tau}
{
	type = notation,
 	name = {\ensuremath{\tau}},
 	description = {},
 	sort = tau
}

\newglossaryentry{iota}
{
	type = notation,
 	name = {\ensuremath{\iota}},
 	description = {},
 	sort = iota
}

\newglossaryentry{rho}
{
	type = notation,
 	name = {\ensuremath{\rho_{x_0}}},
 	description = {},
 	sort = rho
}

\newglossaryentry{Z}
{
	type = notation,
 	name = {\ensuremath{\mathcal{Z}^\pi}},
 	description = {},
 	sort = Zpi
}

\newglossaryentry{Ztau}
{
	type = notation,
 	name = {\ensuremath{\mathcal{Z}^{\pi, \tau}}},
 	description = {},
 	sort = Zpitau
}

\newglossaryentry{Bij}
{
	type = notation,
 	name = {\ensuremath{B_{ij}}},
 	description = {},
 	sort = Bij
}

\newglossaryentry{Bijtau}
{
	type = notation,
 	name = {\ensuremath{B_{ij}^{\tau}}},
 	description = {},
 	sort = Bijtau
}

\newglossaryentry{Up}
{
	type = notation,
 	name = {\ensuremath{U_p}},
 	description = {},
 	sort = Up
}

\newglossaryentry{U}
{
	type = notation,
 	name = {\ensuremath{U}},
 	description = {},
 	sort = U
}

\newglossaryentry{Ppi}
{
	type = notation,
 	name = {\ensuremath{\mathcal{P}^\pi}},
 	description = {},
 	sort = Ppi
}

\newglossaryentry{vx}
{
	type = notation,
 	name = {\ensuremath{v(x)}},
 	description = {},
 	sort = v
}

\newglossaryentry{vpi}
{
	type = notation,
 	name = {\ensuremath{v^\pi}},
 	description = {},
 	sort = vpi
}

\newglossaryentry{E}
{
	type = notation,
 	name = {\ensuremath{E(x)}},
 	description = {},
 	sort = E
}

\newglossaryentry{pi}
{
	type = notation,
 	name = {\ensuremath{\pi}},
 	description = {},
 	sort = pi
}

\newglossaryentry{aij}
{
	type = notation,
 	name = {\ensuremath{a_{i,j}}},
 	description = {},
 	sort = aij
}

\newglossaryentry{Sf}
{
	type = notation,
 	name = {\ensuremath{S_f(k,\pi)}},
 	description = {},
 	sort = Sf
}

\newglossaryentry{Lf}
{
	type = notation,
 	name = {\ensuremath{L_f(k,\pi)}},
 	description = {},
 	sort = Lf
}

\newglossaryentry{Cf}
{
	type = notation,
 	name = {\ensuremath{C_f(k,\pi)}},
 	description = {},
 	sort = Cf
}

\newglossaryentry{Fj}
{
	type = notation,
 	name = {\ensuremath{F_j(x)}},
 	description = {},
 	sort = Fj
}

\newglossaryentry{F}
{
	type = notation,
 	name = {\ensuremath{F(x)}},
 	description = {},
 	sort = Fx
}

\newglossaryentry{Fa}
{
	type = notation,
 	name = {\ensuremath{F_{[a]}(x)}},
 	description = {},
 	sort = Fa
}

\newglossaryentry{alphaa}
{
	type = notation,
 	name = {\ensuremath{\alpha_a}},
 	description = {},
 	sort = alphaa
}

\newglossaryentry{alpha1}
{
	type = notation,
 	name = {\ensuremath{\alpha_1}},
 	description = {},
 	sort = alpha1
}

\newglossaryentry{Uijn}
{
	type = notation,
 	name = {\ensuremath{U_{(i,j),n}}},
 	description = {},
 	sort = Uijn
}

\newglossaryentry{Fnj}
{
	type = notation,
 	name = {\ensuremath{F_{n,j}}},
 	description = {},
 	sort = Fnj
}

\newglossaryentry{Fijnk}
{
	type = notation,
 	name = {\ensuremath{F_{(i,j),(n,k)}}},
 	description = {},
 	sort = Fijnk
}

\newglossaryentry{Cijnk}
{
	type = notation,
 	name = {\ensuremath{C_{(i,j),(n,k)}}},
 	description = {},
 	sort = Cijnk
}

\newglossaryentry{Dijnk}
{
	type = notation,
 	name = {\ensuremath{D_{(i,j),(n,k)}}},
 	description = {},
 	sort = Dijnk
}

\newglossaryentry{box}
{
	type = notation,
 	name = {\ensuremath{\boxplus}},
 	description = {},
 	sort = AAAAAAA
}

\newglossaryentry{HPkc}
{
	type = notation,
 	name = {\ensuremath{\mathrm{HP}_k^c}},
 	description = {},
 	sort = HPkc
}

\newglossaryentry{HPc}
{
	type = notation,
 	name = {\ensuremath{\mathrm{HP}^c}},
 	description = {},
 	sort = HPc
}

\newglossaryentry{BigPi}
{
	type = notation,
 	name = {\ensuremath{\Pi_{\chi}}},
 	description = {},
 	sort = Pichi
}